\documentclass[12pt]{amsart}
\usepackage{a4wide}
\usepackage[all]{xy}
\usepackage{pdfsync}
 \usepackage[colorlinks=false]{hyperref}
\usepackage{amssymb}
\usepackage{amsmath}
\usepackage{autonum}
\usepackage{graphicx}
\usepackage[english]{babel}
\usepackage{epsfig}
\usepackage{multirow}
\usepackage{color}
\usepackage{enumerate}
\usepackage[T1]{fontenc}
\usepackage{aurical, pbsi,calligra}
\swapnumbers

\theoremstyle{definition}
\newtheorem{definition}{Definition}[section]
\newtheorem{remark}[definition]{Remark}

\theoremstyle{plain}
\newtheorem{lemma}[definition]{Lemma}

\newtheorem{theorem}[definition]{Theorem}
\usepackage[all]{xy}
\usepackage{pdfsync}
\usepackage{ytableau}

\begin{document}

\title [Two-phase Stefan problem with nonlinear thermal coefficients] {Two-phase Stefan problem for generalized heat equation with nonlinear thermal coefficients}

\author[T. A. Nauryz]{Targyn A. Nauryz}

\address{Kazakh-British Technical University, Institute of Mathematics and Mathematical Modeling, Almaty, Kazakhstan}

\email{targyn.nauryz@gmail.com}

\author[A. C. Briozzo]{Adriana C. Briozzo}

\address{CONICET-Facultad de Ciencias Empresariales, Universidad Austral, Paraguay 1950, S2000FZF Rosario, Argentina}

\email{ABriozzo@austral.edu.ar}



\thanks{The authors were supported by the grant AP14869306 "Special methods for solving electrical contact problems of the Stefan type and their application to the study of electric arc processes" and the Project 80020210100002 from Austral University, Rosario, Argentina.}

\maketitle

\begin{abstract}
In this article we study a mathematical model of the heat transfer in semi infinite material with a variable cross section, when the radial component of the temperature gradient can be neglected in comparison with the axial component is considered. In particular, the temperature distribution in liquid and solid phases of such kind of body can be modelled by Stefan problem for the generalized heat equation. The method of solution is based on similarity principle, which enables us to reduce generalized heat equation to nonlinear ordinary differential equation. Moreover, we determine temperature solution for two phases and free boundaries which describe the position of boiling and melting interfaces. Existence and uniqueness of the solution is provided by using the fixed point Banach theorem.
\end{abstract}

\section{\textbf{Introduction}} 

The heat transfer Stefan problems with phase change and free boundaries are largely studied in engineering and industries applications. They have wide range in broad variety of field in physical, chemical and economical and etc. phenomena \cite{1}-\cite{7}.

The one or two phase change Stefan problem is a free boundary problem for the classical heat equation which requires the determination of the dynamics of temperature field of the liquid phase (melting problem) or of the solid phase (solidification problem), and the location of melting or freezing interfaces at $x=s(t)$. The classical two-phase Stefan problem for material with cross-section variable domain is widely studied in \cite{8}. The non-classical Stefan problem is governed by heat equations with temperature or time-dependent thermal coefficients with given condition temperature and heat flux condition at fixed face are considered in \cite{9}-\cite{15}. Recently, a non-classical inverse Stefan problem for determination of unknown thermal coefficients is successfully considered \cite{16} and the inverse problem of determining the time-dependent thermal conductivity and the transient temperature satisfying the heat equation with initial data also discussed in \cite{17}. 

A generalized heat equation enables us to describe heat transfer in material with cross-section variable in the case when the radial component o f the temperature gradient can be neglected in comparison with the axial component. Such mathematical model is very useful for some applied problems, in particular, for the dynamics of the heating with phase transformation in finite metal bridge. In this paper we consider phase changing in semi-infinite material where temperature fields in liquid and solid phases and location of boiling, melting interfaces have to be determined. The present work provides existence of similarity type solution of two-phase Stefan problem for generalized heat equation which describe the heat transfer in the variable cross-section body. Temperature field in liquid zone can be modelled as
\begin{equation}\label{eq1}
    c_1(\theta_1)\gamma_1(\theta_1)\dfrac{\partial\theta_1}{\partial t}=\dfrac{1}{r^{\nu}}\dfrac{\partial}{\partial r}\bigg(\lambda_1(\theta_1)r^{\nu}\dfrac{\partial\theta_1}{\partial r}\bigg),\;\;\;\alpha(t)<r<\beta(t),\;\;\;0<\nu<1,
\end{equation}
and solid zone is modelled as
\begin{equation}\label{eq2}
    c_1(\theta_2)\gamma_2(\theta_2)\dfrac{\partial\theta_2}{\partial t}=\dfrac{1}{r^{\nu}}\dfrac{\partial}{\partial r}\bigg(\lambda_2(\theta_2)r^{\nu}\dfrac{\partial\theta_2}{\partial r}\bigg),\;\;\;\beta(t)<r<\infty,\;\;\;0<\nu<1,
\end{equation}
and temperature at $\alpha(t)$ and on melting interface $\beta(t)$
\begin{equation}\label{eq3}
    \theta_1(\alpha(t),t)=\theta_b,
\end{equation}
\begin{equation}\label{eq4}
    \theta_1(\beta(t),t)=\theta_m,
\end{equation}
\begin{equation}\label{eq5}
    \theta_2(\beta(t),t)=\theta_m.
\end{equation}
Then Stefan's conditions are
\begin{equation}\label{eq6*}
    \lambda_1(\theta_b)\dfrac{\partial\theta_1}{\partial r}(\alpha(t),t)=-l_b\gamma_b\dfrac{d\alpha}{dt},
\end{equation}
\begin{equation}\label{eq6}
    -\lambda_1(\theta_m)\dfrac{\partial\theta_1(\beta(t),t)}{\partial r}=-\lambda_2(\theta_m)\dfrac{\partial\theta_2(\beta(t),t)}{\partial r}+l_m\gamma_m\dfrac{d\beta}{dt},
\end{equation}
at infinity 
\begin{equation}\label{eq7}
    \theta_2(\infty,t)=0
\end{equation}
the initial conditions for the free boundaries are
\begin{equation}\label{eq8}
    \alpha(0)=0\qquad \beta(0)=0
\end{equation}
where $\theta_1(r,t)$ and $\theta_2(r,t)$ are the temperatures of the liquid and solid phases, $c_i(\theta_i)$, $\gamma_i(\theta_i)$ and $\lambda_i(\theta_i)$ are  specific heat, material's density and thermal conductivity. $\theta_b$ and $\theta_m$ are boiling and melting temperatures where $\theta_b>\theta_m$ and $\alpha(t)$, $\beta(t)$ are the free boundaries which describe location of boiling and melting interfaces. $l_b,\;l_m>0$ are latent heat of boiling and melting,respectively, $\gamma_b>0$ and $\gamma_m>0$ are the constant density of mass at the boiling and melting temperature. 

The aim of this paper is provide the solution of heat transfer process with generalized heat equation and to show that similarity type solution exists and it is unique. In Section 2 we provide the existence of at least one similarity solution to the problem \eqref{eq1}–\eqref{eq8} with boiling and melting interfaces at free boundaries. By using the similarity substitution, an equivalent nonlinear ordinary differential problem and then integral equations of Volterra type are obtained. In Section 3 we proved the existence of solution to the Stefan problem by using fixed point Banach theorem and we solve two coupled equations for the coefficients that characterize the free boundaries.   

\section{\textbf{Similarity solution of the two-phase Stefan problem for generalized heat equation}}
If we use the following dimensionless transformation 
\begin{equation}\label{eq9}
T(r,t)=\dfrac{\theta(r,t)-\theta_m}{\theta_b-\theta_m}
\end{equation}

then problem \eqref{eq1},\eqref{eq2},\eqref{eq3},\eqref{eq4},\eqref{eq5},\eqref{eq6*}, \eqref{eq6},\eqref{eq7} and \eqref{eq8} becomes
\begin{equation}\label{eq10}
    \Tilde{c_1}(T_1)\Tilde{\gamma_1}(T_1)\dfrac{\partial T_1}{\partial t}=\dfrac{1}{r^{\nu}}\dfrac{\partial}{\partial r}\bigg(\Tilde{\lambda_1}(T_1)r^{\nu}\dfrac{\partial T_1}{\partial r}\bigg),\;\;\; \alpha(t)<r<\beta(t),\;\;t>0,
\end{equation}
\begin{equation}\label{eq11}
    \Tilde{c_2}(T_2)\Tilde{\gamma_2}(T_2)\dfrac{\partial T_2}{\partial t}=\dfrac{1}{r^{\nu}}\dfrac{\partial}{\partial r}\bigg(\Tilde{\lambda_1}(T_2)r^{\nu}\dfrac{\partial T_2}{\partial r}\bigg),\;\;\; \beta(t)<r<\infty,\;\;t>0,
\end{equation}
\begin{equation}\label{eq12}
    T_1(\alpha(t),t)=1,
\end{equation}
\begin{equation}\label{eq13}
    T_1(\beta(t),t)=T_2(\beta(t),t)=0
\end{equation}
\begin{equation}\label{eq15*}
    \Tilde{\lambda_1}(1)\dfrac{\partial T_1}{\partial r}\bigg|_{r=\alpha(t)}=-l_b\gamma_b\alpha'(t)/(\theta_b-\theta_m),
\end{equation}
\begin{equation}\label{eq15}
    -\Tilde{\lambda_1}(0)\dfrac{\partial T_1}{\partial r}\bigg|_{r=\beta(t)}=-\Tilde{\lambda_2}(0)\dfrac{\partial T_2}{\partial r}\bigg|_{r=\beta(t)}+l_m\gamma_m\beta'(t)/(\theta_b-\theta_m),
\end{equation}
\begin{equation}\label{eq16}
    T_2(\infty,t)=T_{\infty},
\end{equation}
where $T_{\infty}=-\theta_m/(\theta_b-\theta_m)$ and
$$\Tilde{c_i}(T_i)=c_i((\theta_b-\theta_m)T_i+\theta_m),\;\;\Tilde{\gamma_i}(T_i)=\gamma_i((\theta_b-\theta_m)T_i+\theta_m),\;\;\Tilde{\lambda_i}(T_i)=\lambda_i((\theta_b-\theta_m)T_i+\theta_m).$$
To solve problem \eqref{eq10}-\eqref{eq16} we define the following similarity transformation
\begin{equation}\label{eq17}
    T(r,t)=u(\eta),\;\;\;\eta=\dfrac{r}{2\sqrt{t}},
\end{equation}
and free boundary conditions \eqref{eq13}-\eqref{eq15} imply that the free boundaries $\alpha(t)$ and $\beta(t)$ must be described by
\begin{equation}\label{eq18}
    \alpha(t)=2\alpha_0\sqrt{t},\;\;\;\beta(t)=2\beta_0\sqrt{t},
\end{equation}
where $\alpha_0$ and $\beta_0$ has to be determined.\\
After using \eqref{eq17} we obtain the following problem
\begin{equation}\label{eq19}
    [L_1(u_1)\eta^{\nu}u_1']'+2\eta^{\nu+1}N_1(u_1)u_1'=0,\;\;\;\alpha_0<\eta<\beta_0;
\end{equation}
\begin{equation}\label{eq20}
    [L_2(u_2)\eta^{\nu}u_2']'+2\eta^{\nu+1}N_2(u_2)u_2'=0,\;\;\;\beta_0<\eta<\infty;
\end{equation}
\begin{equation}\label{eq21}
    u_1(\alpha_0)=1,
\end{equation}
\begin{equation}\label{eq22}
    u_1(\beta_0)=0,
\end{equation}
\begin{equation}\label{eq23}
    u_2(\beta_0)=0,
\end{equation}
\begin{equation}\label{eq24*}
    L_1(u_1(\alpha_0))u_1'(\alpha_0)=-2l_b\gamma_b\alpha_0/(\theta_b-\theta_m),
\end{equation}
\begin{equation}\label{eq24}
    L_1(u_1(\beta_0))u_1'(\beta_0)=L_2(u_2(\beta_0)) u_2'(\beta_0)-2l_m\gamma_m\beta_0/(\theta_b-\theta_m),
\end{equation}
\begin{equation}\label{eq25}
    u_2(\infty)=u_c,
\end{equation}
where $u_c=-\theta_m/(\theta_b-\theta_m)$ and
$$L_i(u_i)=\lambda_i((\theta_b-\theta_m)u_i+\theta_m),\;\;N_i(u_i)=c_i((\theta_b-\theta_m)u_i+\theta_m)\cdot\gamma_i((\theta_b-\theta_m)u_i+\theta_m),\;\;i=1,2.$$
To solve nonlinear differential equation
$$[L_i(u_i)\eta^{\nu}u_i']'+2\eta^{\nu+1}N_i(u_i)u_i'=0,\;\;i=1,2$$
we use substitution
\begin{equation}\label{eq26}
    L_i(u_i)\eta^{\nu}u_i'(\eta)=v_i(\eta),\;\;i=1,2
\end{equation}
and using condition \eqref{eq21},\eqref{eq22}, \eqref{eq23} and \eqref{eq25}  then we obtain solutions for liquid and solid phases
\begin{equation}\label{eq27}
    u_1(\eta)=1-\dfrac{\Phi_1[\alpha_0,\eta,L_1(u_1),N_1(u_1)]}{\Phi_1[\alpha_0,\beta_0, L_1(0), N_1(0)]},\;\;\;\alpha_0\leq\eta\leq\beta_0;
\end{equation}
\begin{equation}\label{eq28}
    u_2(\eta)=u_{c}\dfrac{\Phi_2[\beta_0,\eta,L_2(u_2),N_2(u_2)]}{\Phi_2[\beta_0,\infty,L_2(u_c),N_2(u_c)]}, \;\;\;\beta_0\leq\eta<\infty;
\end{equation}
which satisfies the equations \eqref{eq19} and \eqref{eq20} where
$$\Phi_1[\alpha_0,\eta,L_1(u_1),N_1(u_1)]=\int\limits_{\alpha_0}^{\eta}\dfrac{E_1[\alpha_0,s,u_1]}{s^{\nu}L_1(u_1(s))}ds,$$
$$\Phi_2[\beta_0,\eta,L_2(u_2),N_2(u_2)]=\int\limits_{\beta_0}^{\eta} \dfrac{E_2[\beta_0,s,u_2]}{s^{\nu}L_2(u_2(s))}ds,$$
$$E_1[\alpha_0,\eta, u_1]=\exp\Bigg(-2\int\limits_{\alpha_0}^{\eta}s\dfrac{N_1(u_1(s))}{L_1(u_1(s))}ds\Bigg),$$
$$E_2[\beta_0,\eta, u_2]=\exp\Bigg(-2\int\limits_{\beta_0}^{\eta}s\dfrac{N_2(u_2(s))}{L_2(u_2(s))}ds\Bigg).$$

From condition \eqref{eq24*} we can derive
\begin{equation}\label{eq29*}
    \dfrac{1}{\Phi_1[\alpha_0,\beta_0, L_1(0), N_1(0)]}=\dfrac{2l_b\gamma_b(\alpha_0)^{\nu+1}}{\theta_b-\theta_m}
\end{equation}
and condition \eqref{eq24} becomes
\begin{equation}\label{eq29}
    \dfrac{1}{\Phi_1[\alpha_0,\beta_0, L_1(0), N_1(0)]}+\dfrac{u_c }{\Phi_2[\beta_0,\infty,L_2(u_c),N_2(u_c)]}=\dfrac{2l_m\gamma_m(\beta_0)^{\nu+1}}{\theta_b-\theta_m}
\end{equation}
and $\alpha_0$ and $\beta_0$ can be found from \eqref{eq29*} and \eqref{eq29}.

The solution of free boundary problem \eqref{eq1}-\eqref{eq8} is given by  $$\theta_i(r,t)=\theta_m+(\theta_b-\theta_m)u_i(\eta),\;\;i=1,2,$$ 
where $\eta=r/(2\sqrt{t})$ and  functions $u_i(\eta),\;i=1,2$ and $\alpha_0>0,\;\beta_0>0$ must satisfy the nonlinear integral equations \eqref{eq27} and \eqref{eq28} and expressions \eqref{eq29*},\eqref{eq29}.

\section{\textbf{Existence of solution of the problem}}
To prove existence of solution \eqref{eq27} and \eqref{eq28} we have to obtain some preliminary results. By using the fixed point theorem we will prove that \eqref{eq27} and \eqref{eq28} have a unique solution for any pair $\alpha_0>0, \beta_0>0$. Let we define new notation $\Phi_1[\alpha_0,\eta, u_1]=\Phi_1[\alpha_0,\eta, L_1(u_1), N_1(u_1)]$ and $\Phi_2[\beta_0,\eta, u_2]=\Phi_2[\beta_0,\eta, L_2(u_2), N_2(u_2)]$ for convenient proving. Suppose there exists some positive constants $L_{im},L_{iM},N_{im}$,$N_{iM}$, $\delta$ and $\mu$ such that
\begin{equation}\label{eq30}
    L_{1m}\leq L_1(u_1(\eta))\leq L_{1M},\;\;\;N_{1m}\leq N_1(u_1(\eta))\leq N_{1M}
\end{equation} and 
\begin{equation}\label{eq30*}
    L_{2m}\eta^\mu\leq L_2(u_2(\eta))\leq L_{2M}\eta^\mu,\;\;\;N_{2m}\eta^\delta\leq N_2(u_2(\eta))\leq N_{2M}\eta^\delta 
\end{equation} with
\begin{equation}
\frac{3-\nu}{2}<\mu<2 ,\;\;\;\mu<\delta<3\mu+\nu-3
\end{equation}
and assume that specific heat and dimensionless thermal conductivity are Lipschitz functions, then there exists constants $\tilde{L_i},\;\tilde{N_i}$ such that
\begin{equation}\label{eq31}
    ||L_i(f)-L_i(g)||\leq\tilde{L_i}||f-g||,
\end{equation}
\begin{equation}\label{eq32}
    ||N_i(f)-N_i(g)||\leq\tilde{N_i}||f-g||
\end{equation}
for all $f,g\in C^0(\mathbb{R^{+}})\cap L^{\infty}(\mathbb{R^{+}})$ where $\mathbb{R^{+}}=[0,\infty).$
\begin{lemma}\label{lemma1}
\begin{enumerate}
    \item [a)] For all $\eta\in [\alpha_0, \beta_0]$ we have
    $$\exp\bigg(-\dfrac{N_{1M}}{L_{1m}}(\eta^2-\alpha_0^2)\bigg)\leq E_1[\alpha_0,\eta, u_1]\leq \exp\bigg(-\dfrac{N_{1m}}{L_{1M}}(\eta^2-\alpha_0^2\bigg).$$
    \item[b)] For $\eta>\beta_0$ and $\xi:=2+\delta-\mu>0$ we have
    $$\exp\bigg(-\dfrac{N_{2M}}{L_{2m}\xi}(\eta^{\xi}-\beta_0^{\xi})\bigg)\leq E_2[\beta_0,\eta, u_2]\leq \exp\bigg(-\dfrac{N_{2m}}{L_{2M}\xi}(\eta^{\xi}-\beta_0^{\xi}\bigg).$$
\end{enumerate}
\end{lemma}
\begin{proof}
\begin{enumerate}
\item [a)] $E_1[\alpha_0,\eta, u_1]\leq \exp\Bigg(-2\frac{N_{1m}}{L_{1M}}\int\limits_{\alpha_0}^{\eta}sds\Bigg)=\exp\bigg(-\frac{N_{1m}}{L_{1M}}(\eta^2-\alpha_0^2)\bigg).$
    \item[b)] $E_2[\beta_0,\eta, u_2]\leq \exp\Bigg(-2\frac{N_{2m}}{L_{2M}}\int\limits_{\beta_0}^{\eta}s^{1+\delta-\mu}ds\Bigg)=\exp\bigg(-\frac{N_{2m}}{L_{2M}(2+\delta-\mu)}(\eta^{2+\delta-\mu}-\beta_0^{2+\delta-\mu})\bigg).$
\end{enumerate}
\end{proof}
\begin{lemma}\label{lemma2}
a) For all $\eta\in[\alpha_0,\beta_0]$ we have
$$\frac{1}{2L_{1M}}\exp\bigg(\tfrac{N_{1M}}{L_{1m}}\alpha_0^2\bigg)\sqrt{\frac{N_{1M}^{\nu-1}}{L_{1m}^{\nu-1}}}\bigg[\gamma\bigg(\frac{1-\nu}{2},\eta^2\frac{N_{1M}}{L_{1m}}\bigg)-\gamma\bigg(\frac{1-\nu}{2},  \frac{N_{1M}}{L_{1m}}\alpha_0^2\bigg)\bigg]\leq\Phi_1[\alpha_0,\eta, u_1]$$
$$\leq\frac{1}{2L_{1m}}\exp\bigg(\frac{N_{1m}}{L_M}\alpha_0^2\bigg)\sqrt{\frac{N_{1m}^{\nu-1}}{L_{1M}^{\nu-1}}}\bigg[\gamma\bigg(\frac{1-\nu}{2},\eta^2\frac{N_{1m}}{L_{1M}}\bigg)-\gamma\bigg(\frac{1-\nu}{2},\frac{N_{1m}}{L_{1M}}\alpha_0^2\bigg)\bigg],$$

b) For $\eta>\beta_0$ we have
$$\frac{1}{L_{2M}}\exp\bigg(\tfrac{2\beta_0^{2+\delta-\mu}N_{2M}}{L_{2m}2+(\delta-\mu)}\bigg)\bigg(\tfrac{L_{2m}} {2N_{2M}}\bigg)^\frac{\delta-3\mu+5}{2+\delta-\mu}(2+\delta-\mu)^{\tfrac{1-\delta-\nu}{2+\delta-\mu}} \bigg[\gamma\bigg(\tfrac{3-\nu-\mu}{2+\delta-\mu},\tfrac{2N_{2M}\eta^{2+\delta-\mu}}{L_{2m}(2+\delta-\mu)}\bigg)-\gamma\bigg(\tfrac{3-\nu-\mu}{2+\delta-\mu},\tfrac{2N_{2M}\beta_0^{2+\delta-\mu}}{L_{2m}(2+\delta-\mu)}\bigg)\bigg]$$
$$\leq\Phi_2[\beta_0,\eta;u_2]\leq\frac{1}{L_{2m}}\bigg[\frac{\eta^{1-\mu-\nu}-\beta_0^{1-\mu-\nu}}{1-\mu-\nu}\bigg],$$
where $\gamma(s,x)$ is the incomplete gamma function defined by $\gamma(s,x)=\int\limits_0^x t^{s-1}e^{-t}dt$
\end{lemma}
\begin{proof}
a) $\Phi_1[\alpha_0,\eta,u_1]\leq \frac{1}{L_{1m}}\exp\bigg(\frac{N_{1m}}{L_{1M}}\alpha_0^2\bigg)\int\limits_{\alpha_0}^{\eta}\frac{\exp(-N_{1m}s^2/L_{1M})}{s^{\nu}}ds$ after using substitution $t=s\sqrt{\frac{N_{1m}}{L_{1M}}}$ we obtain
$$\Phi_1[\alpha_0,\eta,u_1]\leq\frac{1}{L_{1m}}\exp\bigg(\frac{N_{1m}}{L_{1M}}\alpha_0^2\bigg)\sqrt{\frac{N_{1m}^{\nu-1}}{L_{1M}^{\nu-1}}}\int\limits_{\alpha_0\sqrt{N_{1m}/L_{1M}}}^{\eta\sqrt{N_{1m}/L_{1M}}}\frac{e^{-t^2}}{t^\nu}dt.$$
Then using substitution $z=t^{1-\nu}$ we get 
$$\Phi_1[\alpha_0,\eta,u_1]\leq \dfrac{1}{L_{1m}(1-\nu)}\exp\bigg(\frac{N_{1m}}{L_{1M}}\alpha_0^2\bigg)\sqrt{\dfrac{N_{1m}^{\nu-1}}{L_{1M}^{\nu-1}}}\int\limits_{(\alpha_0\sqrt{N_m/L_{1M}})^{1-\nu}}^{(\eta\sqrt{N_{1m}/L_{1M}})^{1-\nu}}e^{-z^{\frac{2}{1-\nu}}}dz$$
and taking $y=z^{\frac{2}{1-\nu}}$ then inequality becomes
$$\Phi_1[\alpha_0,\eta,u_1]\leq\frac{1}{L_{1m}(1-\nu)}\exp\bigg(\frac{N_{1m}}{L_{1M}}\alpha_0^2\bigg)\sqrt{\frac{N_{1m}^{\nu-1}}{L_{1M}^{\nu-1}}}\frac{1-\nu}{2}\int\limits_{\alpha_0^2N_m/L_{1M}}^{\eta^2N_{1m}/L_{1M}}y^{\frac{1-\nu}{2}-1}e^{-y}dy.$$
Then by using definition of special function type incomplete gamma function $\gamma(s,x)=\int\limits_0^x t^{s-1}e^{-t}dt$ we have proved that
$$\Phi_1[\alpha_0,\eta, u_1]\leq\frac{1}{2L_{1m}}\exp\bigg(\frac{N_{1m}}{L_{1M}}\alpha_0^2\bigg)\sqrt{\frac{N_{1m}^{\nu-1}}{L_{1M}^{\nu-1}}}\bigg[\gamma\bigg(\frac{1-\nu}{2},\eta^2\frac{N_{1m}}{L_{1M}}\bigg)-\gamma\bigg(\frac{1-\nu}{2}, \frac{N_{1m}}{L_{1M}}\alpha_0^2\bigg)\bigg].$$
The second inequality (b) is proved in a similar way, by using Lemma\ref{lemma1} b) and hypothesis \ref{eq30*}.
\end{proof}
\begin{lemma}\label{lemma3} 
Fixed $\alpha_0, \beta_0\in\mathbb{R^{+}}$ and \eqref{eq31},\eqref{eq32} hold for specific heat and dimensionless thermal conductivity then\\
a) For all $u_1, u_2^*\in C^0[\alpha_0, \eta_0]$ we have
$$|E_1[\alpha_0,\eta,u_1]-E_1[\alpha_0,\eta,u_1^*]|\leq\frac{1}{L_{1m}}\bigg(\tilde{N_1}+\frac{N_{1M}\tilde{L_1}}{L_{1m}}\bigg)(\eta^2-\alpha_0^2)||u_1^*-u_1||.$$\\
b) For all $u_2, u_2^*\in C^0[\beta_0, +\infty)$ we have
$$|E_2[\alpha_0,\eta,u_2]-E_2[\alpha_0,\eta,u_2^*]|\leq2\bigg[\frac{\tilde{N_2}}{L_{2m}}\bigg(\frac{\eta^{2-\mu}-\beta^{2-\mu}}{2-\mu}\bigg)+\frac{\tilde{L_2}N_{2M}}{L^2_{2m}}\bigg(\frac{\eta^{\delta-2\mu+2}-\beta^{\delta-2\mu+2}}{\delta-2\mu+2}\bigg]||u_2^*-u_2||.$$
\end{lemma}
\begin{proof}
a) By using inequality $\exp(-x)-\exp(-y)|\leq |x-y|,\;\;\forall x,y\geq 0$ we get
$$|E_1[\alpha_0,\eta,u_1]-E_1[\alpha_0,\eta,u_1^*]|\leq \Bigg|\exp\Bigg(-2\int\limits_{\alpha_0}^{\eta}s\frac{N_1(u_1(s))}{L_1(u_1(s))}ds\Bigg)-\exp\Bigg(-2\int\limits_{\alpha_0}^{\eta}s\frac{N_1(u_1^*(s))}{L_1(u_1^*(s))}ds\Bigg)\Bigg|$$
$$\leq 2\Bigg|\int\limits_{\alpha_0}^{\eta}s\frac{N_1(u_1)}{L_1(u_1)}ds-\int\limits_{\alpha_0}^{\eta}s\frac{N_1(u_1^*)}{L_1(u_1^*)}ds\Bigg|\leq 2\int\limits_{\alpha_0}^{\eta}\Bigg|\dfrac{N(u_1)}{L(u_1)}-\frac{N_1(u_1^*)}{L_1(u_1^*)}\Bigg|sds$$
$$\leq2\int\limits_{\alpha_0}^{\eta}\Bigg|\frac{N_1(u_1)}{L_1(u_1)}-\frac{N_1(u_1^*)}{L_1(u_1)}+\frac{N_1(u_1^*)}{L_1(u_1)}-\frac{N_1(u_1^*)}{L_1(u_1^*)}\Bigg|sds$$
$$\leq 2\int\limits_{\alpha_0}^{\eta}\Bigg(\dfrac{|N_1(u_1)-N1(u_1^*)|}{|L(u_1)|}+\frac{|L_1(u_1^*)-L_1(u_1)|\cdot|N_1(u_1^*)|}{|L_1(u_1)||L(u_1^*)|}\Bigg)sds$$
$$\leq \frac{2}{L_{1m}}\bigg(\tilde{N_1}+\frac{N_{1M}\tilde{L_1}}{L_{1m}}\bigg)||u_1^*-u_1||\int\limits_{\alpha_0}^{\eta}sds=\frac{1}{L_{1m}}\bigg(\tilde{N_1}+\frac{N_{1M}\tilde{L}}{L_{1m}}\bigg)(\eta^2-\alpha_0^2)||u_1^*-u_1||.$$
The second inequality (b) is proved similarly, we have 
$$|E_2[\beta_0,\eta,u_2]-E_2[\beta_0,\eta,u_2^*]|
\leq2\int\limits_{\beta_0}^{\eta}\Bigg|\frac{N_2(u_2)}{L_2(u_2)}-\frac{N_2(u_2^*)}{L_2(u_2)}+\dfrac{N_2(u_2^*)}{L_2(u_2)}-\frac{N_2(u_2^*)}{L_2(u_2^*)}\Bigg|sds$$
$$\leq 2\int\limits_{\beta_0}^{\eta}\Bigg(\frac{|N_2(u_2)-N_2(u_2^*)|}{|L_2(u_2)|}+\frac{|L_2(u_2^*)-L_2(u_2)|\cdot|N_2(u_2^*)|}{|L_2(u_2)||L(u_2^*)|}\Bigg)sds$$
$$\leq \frac{2\tilde{N_2}}{L_{2m}}||u_2^*-u_2||\int\limits_{\beta_0}^{\eta}s^{1-\mu}ds+\frac{2\tilde{L_2}N_{2M}}{L^2_{2m}}||u_2^*-u_2||\int\limits_{\beta_0}^{\eta}s^{\delta+1-2\mu}ds.$$
$$=2\bigg[\frac{\tilde{N_2}}{L_{2m}}\bigg(\frac{\eta^{2-\mu}-\beta_0^{2-\mu}}{2-\mu}\bigg)+\frac{\tilde{L_2}N_{2M}}{L^2_{2m}}\bigg(\frac{\eta^{\delta-2\mu+2}-\beta_0^{\delta-2\mu+2}}{\delta-2\mu+2}\bigg]||u_2^*-u_2||$$
\end{proof}
\begin{lemma}\label{lemma4}
If $\alpha_0, \beta_0\in\mathbb{R^{+}}$ are given and \eqref{eq31},\eqref{eq32} hold then\\
a) For all $u_1,u_1^*\in C^0[\alpha_0, \beta_0]$ we have
$$|\Phi_1[\alpha_0,\eta,u_1]-\Phi_1[\alpha_0,\eta,u_1^*]|\leq\Tilde{\Phi}_1(\alpha_0,\beta_0)||u_1^*-u_1||,$$
$$\forall\eta\in (\alpha_0, \beta_0).$$
b) For all $u_2,u_2^*\in C^0[\beta_0,+\infty)$ we get
$$|\Phi_2[\beta_0,\eta,u_2]-\Phi_2[\beta_0,\eta,u_2^*]|\leq\Tilde{\Phi}_2(\beta_0,\eta)||u_2^*-u_2||,\;\;\;\forall\eta\in(\beta_0,+\infty),$$
where
\begin{equation}
\Tilde{\Phi}_1(\alpha_0,\eta)=\dfrac{1}{L_m^2}\bigg(\bigg(\tilde{N_1}+\frac{N_{1M}\tilde{L_1}}{L_{1m}}\bigg)\bigg[\frac{\eta^{3-\nu}}{3-\nu}-\alpha_0^2\dfrac{\eta^{1-\nu}}{1-\nu}+\frac{2\alpha_0^2}{(3-\nu)(1-\nu)}\bigg]+\tilde{L_1}\frac{\eta^{1-\nu}}{1-\nu}\bigg),
\end{equation}
\begin{equation} \Tilde{\Phi}_2(\beta_0,\eta)=\bigg[\tfrac{2}{L_{2m}}\bigg[\frac{\tilde{N_2}}{L_{2m}}\bigg(\frac{\eta^{3-2\mu-\nu}-\beta_0^{3-2\mu-\nu}}{(2-\mu)(3-2\mu-\nu)}-\frac{\beta_0^{2-\mu}\eta^{1-\mu-\nu}-\beta_0^{3-2\mu-\nu}}{(2-\mu)(1-\mu-\nu)}\bigg)
\end{equation}
\[+ \frac{\tilde{L_2}N_{2M}}{L^2_{2m}}\bigg(\frac{\eta^{\delta-3\mu+3-\nu}-\beta_0^{\delta-3\mu+3-\nu}}{(\delta-2\mu+2)(\delta-3\mu+3-\nu)}-\frac{\beta_0^{\delta-2\mu+2}\eta^{1-\mu-\nu}-\beta_0^{\delta-3\mu+3-\nu}}{(\delta-2\mu+2)(1-\mu-\nu)}\bigg)\bigg]\]
\[+\frac{\tilde{L_2}}{L_{2m}^2}\bigg(\frac{\eta^{1-\mu-\nu}-\beta^{1-\mu-\nu}}{{1-\mu-\nu}}\bigg)\bigg]
\]
\end{lemma}
\begin{proof}
a) By using lemmas \ref{lemma2} and \ref{lemma3} for (a) we obtain 
$$|\Phi_1[\alpha_0,\eta,u_1]-\Phi_1[\alpha_0,\eta,u_1^*]|\leq T_1(\eta)+T_2(\eta)$$
where 
$$T_1(\eta)\equiv\int\limits_{\alpha_0}^{\eta}\dfrac{|E_1[\alpha_0,\eta,u_1]-E_1[\alpha_0,\eta, u_1^*]|}{s^{\nu}L_1(u_1(s))}ds\leq \frac{1}{L_{1m}^2}\bigg(\tilde{N_1}+\frac{N_{1M}\tilde{L_1}}{L_{1m}}\bigg)||u_1^*-u_1||\int\limits_{\alpha_0}^{\eta}(s^2-\alpha_0^2)s^{-\nu}ds$$
$$\leq \frac{1}{L_{1m}^2}\bigg(\tilde{N_1}+\dfrac{N_{1M}\tilde{L_1}}{L_{1m}}\bigg)\bigg[\frac{\eta^{3-\nu}}{3-\nu}-\alpha_0^2\frac{\eta^{1-\nu}}{1-\nu}+\frac{2\alpha_0^2}{(3-\nu)(1-\nu)}\bigg]||u_1^*-u_1||$$
and
$$T_2(\eta)\equiv \int\limits_{\alpha_0}^{\eta}\bigg|\frac{1}{L_1(u_1)}-\frac{1}{L_1(u_1^*)}\bigg|\frac{1}{s^{\nu}}\exp\Bigg(-2\int\limits_{\alpha_0}^{\eta}t\frac{N_1(u_1^*)}{L_1(u_1^*)}dt\Bigg)ds$$
$$\leq \int\limits_{\alpha_0}^{\eta}\frac{|L_1(u_1^*)-L_1(u_1)|}{|L_1(u_1)||L_1(u_1^*)|}\frac{ds}{s^{\nu}}\leq \frac{\tilde{L_1}}{L_{1m}^2}||u_1^*-u_1||\int\limits_{\alpha_0}^{\eta}\frac{ds}{s^{\nu}}\leq \frac{\tilde{L_1}\eta^{1-\nu}}{L_{1m}^2(1-\nu)}||u_1^*-u_1||. $$
Finally we get
$$T_1(\eta)+T_2(\eta)\leq \dfrac{1}{L_{1m}^2}||u_1^*-u_1||\bigg(\bigg(\tilde{N_1}+\dfrac{N_M\tilde{L}}{L_{1m}}\bigg)\bigg[\frac{\eta^{3-\nu}}{3-\nu}-\alpha_0^2\dfrac{\eta^{1-\nu}}{1-\nu}+\frac{2\alpha_0^2}{(3-\nu)(1-\nu)}\bigg]+\tilde{L}\frac{\eta^{1-\nu}}{1-\nu}\bigg).$$
We can prove (b) analogously, we consider 
$$|\Phi_2[\beta_0,\eta,u_2]-\Phi_2[\beta_0,\eta,u_2^*]|\leq U_1(\eta)+U_2(\eta)$$ where 
$$U_1(\eta)\equiv\int\limits_{\beta_0}^{\eta}\frac{|E_2[\beta_0,\eta,u_2]-E_2[\beta_0,\eta, u_2^*]|}{s^{\nu}L_2(u_2(s))}ds$$
$$\leq \tfrac{1}{L_{2m}}\int\limits_{\beta_0}^{\eta}2\bigg[\dfrac{\tilde{N_2}}{L_{2m}}\bigg(\frac{s^{2-\mu}-\beta_0^{2-\mu}}{2-\mu}\bigg)+\frac{\tilde{L_2}N_{2M}}{L^2_{2m}}\bigg(\frac{s^{\delta-2\mu+2}-\beta_0^{\delta-2\mu+2}}{\delta-2\mu+2}\bigg]s^{-\nu-\mu}ds||u_2^*-u_2||$$
$$\leq \tfrac{2}{L_{2m}}\bigg[\frac{\tilde{N_2}}{L_{2m}}\bigg(\frac{\eta^{3-2\mu-\nu}-\beta_0^{3-2\mu-\nu}}{(2-\mu)(3-2\mu-\nu)}-\frac{\beta_0^{2-\mu}\eta^{1-\mu-\nu}-\beta_0^{3-2\mu-\nu}}{(2-\mu)(1-\mu-\nu)}\bigg)\bigg]||u_2^*-u_2||$$
$$+ \tfrac{2}{L_{2m}}\bigg[\frac{\tilde{L_2}N_{2M}}{L^2_{2m}}\bigg(\frac{\eta^{\delta-3\mu+3-\nu}-\beta_0^{\delta-3\mu+3-\nu}}{(\delta-2\mu+2)(\delta-3\mu+3-\nu)}-\frac{\beta_0^{\delta-2\mu+2}\eta^{1-\mu-\nu}-\beta_0^{\delta-3\mu+3-\nu}}{(\delta-2\mu+2)(1-\mu-\nu)}\bigg]||u_2^*-u_2||$$

and $$U_2(\eta)\equiv \int\limits_{\beta_0}^{\eta}\bigg|\frac{1}{L_2(u_2)}-\frac{1}{L_2(u_2^*)}\bigg|\frac{E(\beta_0,s,(u_2^*)}{s^{\nu}}ds
\leq \int\limits_{\beta_0}^{\eta}\frac{|L_2(u_2^*)-L_2(u_2)|}{|L_2(u_2)||L_2(u_2^*)|}\frac{ds}{s^{\nu}}$$ $$\leq \frac{\tilde{L_2}}{L_{2m}^2}||u_2^*-u_2||\int\limits_{\beta_0}^{\eta}\frac{ds}{s^{\nu+\mu}}\leq \frac{\tilde{L_2}}{L_{2m}^2}\bigg(\frac{\eta^{1-\mu-\nu}-\beta^{1-\mu-\nu}}{{1-\mu-\nu}}\bigg)||u_2^*-u_2||. $$
Then $$U_1(\eta)+U_2(\eta)\leq\bigg[\tfrac{2}{L_{2m}}\bigg[\frac{\tilde{N_2}}{L_{2m}}\bigg(\frac{\eta^{3-2\mu-\nu}-\beta_0^{3-2\mu-\nu}}{(2-\mu)(3-2\mu-\nu)}-\frac{\beta_0^{2-\mu}\eta^{1-\mu-\nu}-\beta_0^{3-2\mu-\nu}}{(2-\mu)(1-\mu-\nu)}\bigg)$$
$$+ \frac{\tilde{L_2}N_{2M}}{L^2_{2m}}\bigg(\frac{\eta^{\delta-3\mu+3-\nu}-\beta_0^{\delta-3\mu+3-\nu}}{(\delta-2\mu+2)(\delta-3\mu+3-\nu)}-\frac{\beta_0^{\delta-2\mu+2}\eta^{1-\mu-\nu}-\beta_0^{\delta-3\mu+3-\nu}}{(\delta-2\mu+2)(1-\mu-\nu)}\bigg)\bigg]$$
$$+\frac{\tilde{L_2}}{L_{2m}^2}\bigg(\frac{\eta^{1-\mu-\nu}-\beta^{1-\mu-\nu}}{{1-\mu-\nu}}\bigg)\bigg]||u_2^*-u_2||. $$

\end{proof}

\begin{theorem}\label{th1}
Let given $\alpha_0,\beta_0\in\mathbb{R}$.  Suppose \eqref{eq31} and \eqref{eq32} hold. If the following inequality
\begin{equation}\label{eq33}
   \epsilon(\alpha_0,\beta_0)=\dfrac{2L_{1M}^{\frac{5-\nu}{2}}L_{1m}^{\nu-2}\Tilde{\Phi}_1(\alpha_0,\beta_0)}{N_{1m}^{\frac{1-\nu}{2}}\bigg[\gamma\bigg(\tfrac{1-\nu}{2},\beta_0^2\tfrac{N_{1M}}{L_{1m}}\bigg)-\gamma\bigg(\tfrac{1-\nu}{2},\alpha_0^2\tfrac{N_{1M}}{L_{1m}}\bigg)\bigg]}<1,
\end{equation}
where 
$$\Tilde{\Phi}_1(\alpha_0,\beta_0)=\dfrac{1}{L_{1m}^{2}}\bigg(\bigg(\tilde{N_1}+\dfrac{N_{1M}\tilde{L}}{L_{1m}}\bigg)\bigg[\dfrac{\beta_0^{3-\nu}}{3-\nu}-\alpha_0^2\dfrac{\beta_0^{1-\nu}}{1-\nu}+\dfrac{2\alpha_0^2}{(3-\nu)(1-\nu)}\bigg]+\tilde{L_1}\dfrac{\beta_0^{1-\nu}}{1-\nu}\bigg),$$
is satisfied, then there exists unique solution $u_1\in C^0[\alpha_0,\beta_0]$ of integral equation \eqref{eq27}.
\end{theorem}
\begin{proof} Consider the Banach space $C^0[\alpha_0,\beta_0]$ of continuos real valued functions endowed with the norm $||u||=\max_{\eta\in[\alpha_0,\beta_0]}|u(\eta)|$. 
Let $V: C^0[\alpha_0,\beta_0]\to C^0[\alpha_0, \beta_0]$ the operator defined by
$$V(u_1)(\eta)=1-\dfrac{\Phi_1[\alpha_0,\eta,u_1]}{\Phi_1[\alpha_0,\beta_0,u_1]},\;\;\;u_1\in C^0[\alpha_0, \beta_0].$$
We will use fixed point Banach theorem to prove that for each fixed interval $[\alpha_0,\beta_0]$ there exists a unique $u_1 \in C^0[\alpha_0,\beta_0]$  such that
$$V(u_1)(\eta)=u_1(\eta),\;\;\alpha_0\leq\eta\leq\beta_0.$$ We show that $V$ is a contracting self-map of $C^0[\alpha_0,\beta_0]$. \\
Let $u_1,u_1^*\in C^0[\alpha_0,\beta_0]$, we have
$$||V(u_1)-V(u_1^*)||=\max_{\eta\in[\alpha_0,\beta_0]}|V(u_1)(\eta)-V(u_1^*)(\eta)|\leq \max_{\eta\in[\alpha_0,\beta_0]}\bigg|\dfrac{\Phi_1[\alpha_0,\eta,u_1^*]}{\Phi_1[\alpha_0,\beta_0,u_1^*]}-\dfrac{\Phi_1[\alpha_0,\eta, u_1(\eta)]}{\Phi_1[\alpha_0,\beta_0,u_1]}\bigg|$$
$$\leq \max_{\eta\in[\alpha_0,\beta_0]}\bigg(\dfrac{|\Phi_1[\alpha_0,\eta,u_1^*]\Phi_1[\alpha_0,\beta_0,u_1]-\Phi_1[\alpha_0,\eta,u_1]\Phi_1[\alpha_0,\beta_0,u_1^*]|}{|\Phi_1[\alpha_0,\beta_0,u_1^*|\cdot|\Phi_1[\alpha_0,\beta_0,u_1]|}\bigg)$$
$$\leq A(\alpha_0,\beta_0)\max_{\eta\in[\alpha_0,\beta_0]}\bigg(|\Phi_1[\alpha_0,\eta,u_1^*]\Phi_1[\alpha_0,\beta_0,u_1(]-\Phi_1[\alpha_0,\beta_0,u_1^*]\Phi_1[\alpha_0,\eta,u_1^*]|$$
$$+|\Phi_1[\alpha_0,\beta_0,u_1^*]\Phi_1[\alpha_0,\eta,u_1^*]-\Phi_1[\alpha_0,\eta,u_1]\Phi_1[\alpha_0,\beta_0,u_1^*]\bigg)$$
$$\leq A(\alpha_0,\beta_0)\max_{\eta\in[\alpha_0,\beta_0]}\bigg(|\Phi_1[\alpha_0,\eta, u_1^*]|\cdot|\Phi_1[\alpha_0,\beta_0, u_1]-\Phi_1[\alpha_0,\beta_0, u_1^*]|$$
$$+|\Phi_1[\alpha_0,\beta_0,u_1^*]|\cdot|\Phi_1[\alpha_0,\eta, u_1^*]-\Phi_1[\alpha_0,\eta, u_1]\bigg)$$
where 
$$A(\alpha_0,\beta_0)=\dfrac{4L_{1M}^2 L_{1m}^{\nu-1}}{ N_{1M}^{\nu-1}\exp\bigg(\frac{2\alpha_0^2N_{1M}}{L_{1m}}\bigg) \bigg(\gamma\bigg(\tfrac{1-\nu}{2},\beta_0^2\frac{N_{1M}}{L_{1m}}\bigg)-\gamma\bigg(\tfrac{1-\nu}{2},\alpha_0^2\frac{N_{1M}}{L_{1m}}\bigg)\bigg)^2}>0.$$
Then, we obtain 
$$||V(u_1)-V(u_1^*)||\leq  \tfrac{A(\alpha_0,\beta_0)\sqrt{N_{1m}^{\nu-1}}\exp\bigg(\tfrac{N_{1m}}{L_{1M}}\alpha_0^2\bigg)\Tilde{\Phi}_1(\alpha_0,\beta_0)\bigg[\gamma\bigg(\tfrac{1-\nu}{2},\tfrac{\beta_0^2 N_{1m}}{L_{1M}}\bigg)-\gamma\bigg(\tfrac{1-\nu}{2},\tfrac{\alpha_0^2 N_{1m}}{L_{1M}}\bigg)\bigg]}{2L_{1m}\sqrt{L_{1M}^{\nu-1}}} ||u_1^*-u_1||$$ that is
$$||V(u_1)-V(u_1^*)||\leq  \epsilon(\alpha_0,\beta_0)||u_1^*-u_1||$$

If condition \eqref{eq33} satisfied, then $V$ is contraction operator and by fixed point Banach theorem there exists unique solution $u_1=u_1(\alpha_0, \beta_0)$ of integral equation \eqref{eq27}. 
\end{proof}

\begin{theorem}\label{th2}
Let $\beta_0$ be a given positive real number and \eqref{eq31}, \eqref{eq32} hold. If the following inequality 
\begin{equation}\label{eq34}
        \sigma(\beta_0):=B(\beta_0)  \dfrac{\beta_0^{1-\mu-\nu}\Tilde{\Phi}_2(\beta_0,+\infty)}{L_{2m}(-1+\mu+\nu)}<1,
\end{equation} is satisfied, where 
\begin{equation}\label{fitilde} \Tilde{\Phi}_2(\beta_0,+\infty)=\tfrac{2}{L_{2m}^{2}}\bigg[\tilde{N_2}\tfrac{\beta_0^{3-2\mu-\nu}}{(3-2\mu-\nu)(1-\mu-\nu)}+ \tfrac{\tilde{L_2}N_{2M}}{L_{2m}}\tfrac{\beta_0^{\delta-3\mu+3-\nu}}{(\delta-3\mu+3-\nu)(1-\mu-\nu)}+\tfrac{\tilde{L_2}}{L_{2m}}\tfrac{\beta^{1-\mu-\nu}}{{\mu+\nu-1}}\bigg],
\end{equation}
\begin{equation}\label{b}
B(\beta_0)=\dfrac{|u_c|L^{2}_{2M} (2N_{2M})^{\frac{\delta-3\mu+5}{\xi}}}{\bigg[\exp\bigg(\tfrac{2\beta_0^{\xi}N_{2M}}{L_{2m}\xi}\bigg)L_{2m}^\frac{\delta-3\mu+5}{\xi}(\xi)^{\tfrac{1-\delta-\nu}{\xi}} \bigg[\Gamma\bigg(\tfrac{3-\nu-\mu}{\xi}\bigg)-\gamma\bigg(\tfrac{3-\nu-\mu}{\xi},\tfrac{2N_{2M}\beta_0^{\xi}}{L_{2m}\xi}\bigg)\bigg]\bigg]^2}>0.\end{equation}
with $\xi=2+\delta-\mu$ and $\Gamma(a)=\int\limits_0^{+\infty} t^{a-1}e^{-t}dt$,
then there exists unique solution $u_2\in C^0[\beta_0,+\infty)$ of integral equation \eqref{eq28}.
\end{theorem}
\begin{proof} Let $$K=\big\lbrace{u\in C^0[\beta_0,+\infty)/u(\beta_0)=0, u(+\infty)=u_c\big\rbrace}$$ a closed subset of the continuous and bounded real valued functions on $[\beta_0,+\infty)$ endowed with the suremum norm $|p|u||=\sup_{\eta\in[\beta_0,+\infty)}|u(\eta)|$. 
Let $W: K\to K$ is operator defined by 
$$W(u_2)(\eta)=u_c\dfrac{\Phi_2[\beta_0,\eta,u_2]}{\Phi_2[\beta_0,+\infty,u_2]},\;\;\;u_2\in K$$
we will prove that for each $\beta_0$ fixed there exists a unique fixed point of operator $W$ that is 
$$W(u_2)(\eta)=u_2(\eta),\;\;\;\beta_0<\eta.$$
To show that $W$ is a contracting self-map of $K$ let  $u_2,u_2^*\in K$ then
$$||W(u_2)-W(u_2^*)||=\sup_{\eta\in[\beta_0,+\infty)}|W(u_2)(\eta)-W(u_2^*)(\eta)|\leq \sup_{\eta\in[\beta_0,+\infty)}\bigg|u_c\frac{\Phi_2[\beta_0,\eta,u_2]}{\Phi_2[\beta_0,+\infty,u_2]}-u_c\frac{\Phi_2[\beta_0,\eta,u_2^*]}{\Phi_2[\beta_0,+\infty,u_2^*]}\bigg|$$
$$\leq |u_c|\sup_{\eta\in[\beta_0,+\infty)}\bigg(\frac{|\Phi_2[\beta_0,\eta,u_2]\Phi_2[\beta_0,+\infty,u_2^*]-\Phi_2[\beta_0,\eta,u_2^*]\Phi_2[\beta_0,+\infty,u_2]|}{|\Phi_2[\beta_0,+\infty,u_2]||\Phi_2\beta_0,[+\infty,u_2]|}\bigg)$$
$$\leq B(\beta_0) \sup_{\eta\in[\beta_0,+\infty)}\bigg( \big|\Phi_2[\beta_0,\eta,u_2]\Phi_2[+\infty,u_2^*]-\Phi_2[\beta_0,\eta,u_2^*]\Phi_2[\beta_0,+\infty,u_2^*]$$
$$+\Phi_2[\beta_0,\eta,u_2^*]\Phi_2[\beta_0,+\infty,u_2^*]-\Phi_2[\beta_0,\eta,u_2^*]\Phi_2[\beta_0,+\infty,u_2]\big|\big)$$
$$\leq B(\beta_0) \sup_{\eta\in[\beta_0,+\infty)} \bigg(\frac{\Tilde{\Phi}_2(\beta_0,\eta)}{L_{2m}}\bigg[\frac{\beta_0^{1-\mu-\nu}}{-1+\mu+\nu}\bigg]+\frac{\Tilde{\Phi}_2(\beta_0,+\infty)}{L_{2m}}\bigg[\frac{\eta^{1-\mu-\nu}-\beta_0^{1-\mu-\nu}}{1-\mu-\nu}\bigg]\bigg)||u_2^*-u_2||$$
$$\leq B(\beta_0)  \dfrac{\beta_0^{1-\mu-\nu}\Tilde{\Phi}_2(\beta_0,+\infty)}{L_{2m}(-1+\mu+\nu)} ||u_2^*-u_2||=\sigma(\beta_0)||u_2^*-u_2||$$

where 
$$B(\beta_0)=\dfrac{|u_c|L^{2}_{2M} (2N_{2M})^{\frac{\delta-3\mu+5}{\xi}}}{\bigg[\exp\bigg(\tfrac{2\beta_0^{\xi}N_{2M}}{L_{2m}\xi}\bigg)L_{2m}^\frac{\delta-3\mu+5}{\xi}(\xi)^{\tfrac{1-\delta-\nu}{\xi}} \bigg[\Gamma\bigg(\tfrac{3-\nu-\mu}{\xi}\bigg)-\gamma\bigg(\tfrac{3-\nu-\mu}{\xi},\tfrac{2N_{2M}\beta_0^{\xi}}{L_{2m}\xi}\bigg)\bigg]\bigg]^2}>0$$
with $\xi=2+\delta-\mu$ and $\Gamma(a)=\int\limits_0^{+\infty} t^{a-1}e^{-t}dt$

If inequality \eqref{eq34} is satisfied, then $W$ is contraction operator and there exists unique solution $u_2$ of \eqref{eq28}. 
\end{proof}
\begin{remark}
The solution given by Th\eqref{th1} depends on $(\alpha_0, \beta_0)$ that is  $u_1=u_1(\alpha_0, \beta_0)$ and the solution $u_2$ given by Th\eqref{th2} depends on $\beta_0$ this is $u_2=u_2(\beta_0)$
\end{remark}

Next, we will analyze the existence of coefficients $\alpha_0, \beta_0$ with $0<\alpha_0<\beta_0$, such that the inequalities \eqref{eq33} and \eqref{eq34} hold. First we study $\sigma=\sigma(\beta_0)$ given by 
$$ \sigma(\beta_0)=B(\beta_0)  \dfrac{\beta_0^{1-\mu-\nu}\Tilde{\Phi}_2(\beta_0,+\infty)}{L_{2m}(-1+\mu+\nu)}, \beta_0>0$$

\begin{lemma}\label{lemmasigma} Function $\sigma=\sigma(\beta_0)$ satisfies $\sigma(0)=+\infty$ and $\sigma(+\infty)=0$ then there exists $\tilde{\beta_0}$ such that $\sigma(\tilde{\beta_0})=1$ and $\sigma(\beta_0)<1$ for all $\beta_0>\tilde{\beta_0}$
\end{lemma}
\begin{proof}
From definition $\eqref{b}$ we have
$$B(0)=\dfrac{|u_c|L^{2}_{2M} (2N_{2M})^{\frac{\delta-3\mu+5}{2+\delta-\mu}}}{\bigg[L_{2m}^\frac{\delta-3\mu+5}{2+\delta-\mu}(\xi)^{\tfrac{1-\delta-\nu}{\xi}} \Gamma\bigg(\tfrac{3-\nu-\mu}{\xi}\bigg)\bigg]^2}$$
To obtain $B(+\infty)$ we analize $$\lim\limits_{\beta_0\to +\infty} \dfrac{\exp\bigg(-\tfrac{2\beta_0^{\xi}N_{2M}}{L_{2m}\xi}\bigg)}{\Gamma\bigg(\tfrac{3-\nu-\mu}{\xi}\bigg)-\gamma\bigg(\tfrac{3-\nu-\mu}{\xi},\tfrac{2N_{2M}\beta_0^{\xi}}{L_{2m}(2+\delta-\mu)}\bigg)}=\lim\limits_{\beta_0\to +\infty} \dfrac{\exp\bigg(-\tfrac{2\beta_0^{2+\delta-\mu}N_{2M}}{L_{2m}\xi}\bigg)\tfrac{2N_{2M}}{L_{2m}}\beta_0^{1+\delta-\mu}}{\frac{d}{d\beta_0}\bigg(\gamma\bigg(\tfrac{3-\nu-\mu}{\xi},\tfrac{2N_{2M}\beta_0^{\xi}}{L_{2m}(2+\delta-\mu)}\bigg)\bigg)}$$
and taking into account
$$\frac{d}{d\beta_0}\bigg(\gamma\bigg(\tfrac{3-\nu-\mu}{\xi},\tfrac{2N_{2M}\beta_0^{\xi}}{L_{2m}\xi}\bigg)\bigg)=\tfrac{2N_{2M}}{L_{2m}}\beta_0^{1+\delta-\mu} \bigg(t^{\tfrac{3-\nu-\mu}{\xi}-1}exp(-t)\bigg)|_{t=\tfrac{2N_{2M}\beta_0^{\xi}}{L_{2m}\xi}}$$
$$=\tfrac{2N_{2M}}{L_{2m}} \beta_0^{1+\delta-\mu}\exp\bigg(-\tfrac{2\beta_0^{\xi}N_{2M}}{L_{2m}(2+\delta-\mu)}\bigg)\bigg(\tfrac{2N_{2M}}{L_{2m}\xi}\bigg)^{\tfrac{1-\nu-\delta}{2+\delta-\mu}}\beta_0^{1-\nu-\delta}$$ and the fact that $1-\nu-\delta<0 $ we have $B(+\infty)=+\infty$
Moreover we have $$F(\beta_0)=\beta_0^{1-\mu-\nu}\Tilde{\Phi}_2(\beta_0,+\infty)$$ satisfies $F(0)=+\infty$ and $F(+\infty)=0$, then $\sigma(0)=+\infty.$

To prove $\sigma(+\infty)=0$ we study 
$$\lim\limits_{\beta_0\to +\infty}\dfrac{F(\beta_0)}{G(\beta_0)}$$ 
where $$G(\beta_0)= \exp\bigg(\tfrac{4\beta_0^{\xi}N_{2M}}{L_{2m}\xi}\bigg)\bigg[\Gamma\bigg(\tfrac{3-\nu-\mu}{\xi}\bigg)-\gamma\bigg(\tfrac{3-\nu-\mu}{\xi},\tfrac{2N_{2M}\beta_0^{2+\delta-\mu}}{L_{2m}\xi}\bigg)\bigg]^{2}$$
We have
$$F'(\beta_0)= a \beta_0^{3-3\mu-2\nu}+b\beta_0^{\delta+3-4\mu-2\nu}+c\beta_0^{1-2\mu-2\nu}$$ with a, b, c constants, and
$$G'(\beta_0)=\tfrac{4N_{2M}}{L_{2m}}\beta_0^{1+\delta-\mu}\exp\bigg(\tfrac{4\beta_0^{\xi}N_{2M}}{L_{2m}\xi}\bigg)\bigg[\Gamma\bigg(\tfrac{3-\nu-\mu}{\xi}\bigg)-\gamma\bigg(\tfrac{3-\nu-\mu}{\xi},\tfrac{2N_{2M}\beta_0^{\xi}}{L_{2m}\xi}\bigg)\bigg].H(\beta_0)$$ where 
$$H(\beta_0)=
-\exp\bigg(-\tfrac{2\beta_0^{\xi}N_{2M}}{L_{2m}\xi}\bigg)\bigg(\tfrac{2N_{2M}}{L_{2m}\xi}\bigg)^{\tfrac{1-\nu-\delta}{\xi}\beta_0^{1-\nu-\delta} +\Gamma\bigg(\tfrac{3-\nu-\mu}{\xi}\bigg)-\gamma\bigg(\tfrac{3-\nu-\mu}{\xi},\tfrac{2N_{2M}\beta_0^{\xi}}{L_{2m}\xi}\bigg)}$$
 
Then $$\frac{F'(\beta_0)}{G'(\beta_0)}=\tfrac{\bigg(a \beta_0^{2-2\mu-2\nu-\delta}+b \beta_0^{+3-3\mu-2\nu}+c\beta_0^{-\delta-\mu-2\nu}\bigg)H(\beta_0)}{\tfrac{4N_{2M}}{L_{2m}}\exp\bigg(\tfrac{4\beta_0^{\xi}N_{2M}}{L_{2m}\xi}\bigg)\bigg[\Gamma\bigg(\tfrac{3-\nu-\mu}{\xi}\bigg)-\gamma\bigg(\tfrac{3-\nu-\mu}{\xi},\tfrac{2N_{2M}\beta_0^{\xi}}{L_{2m}\xi}\bigg)\bigg]}. $$
From hypothesis \eqref{eq30*} we have the numerator goes to zero when $\beta_0$ goes to $+\infty$ and for other hand the denominator goes to $+\infty$. Then we obtain $\sigma(+\infty)=0$.
If $\sigma(0)=+\infty$ and  $\sigma(+\infty)=0$ then there exists $\tilde{\beta_0}>0$ such that $\sigma(\tilde{\beta_0})=1$ and $\sigma(\beta_0)<1$ for all $\beta_0>\tilde{\beta_0}.$
\end{proof}

Now, for each $\beta_0>\tilde{\beta_0}$ we want to find $0<\alpha_0<\beta_0$ such that $\epsilon(\alpha_0,\beta_0)<1$. 

\begin{lemma}\label{lemmaepsilon} For each $\beta_0>\tilde{\beta_0}$ function $\epsilon_{\beta_0}(\alpha_{0}):=\epsilon(\alpha_{0},\beta_0)$ given by \eqref{eq33} for  $0<\alpha_0<\beta_0$ satisfies $$\epsilon_{\beta_0}(0)=\dfrac{2L_{1M}^{\frac{5-\nu}{2}}L_{1m}^{\nu-2}\Tilde{\Phi}_1(0,\beta_0)}{N_{1m}^{\frac{1-\nu}{2}}\gamma\bigg(\tfrac{1-\nu}{2},\beta_0^2\tfrac{N_{1M}}{L_{1m}}\bigg)}, \quad \epsilon_{\beta_0}(\beta_0)=+\infty.$$
Moreover, if
\begin{equation}\label{des}
    \dfrac{2L_{1M}^{\frac{5-\nu}{2}}L_{1m}^{\nu-2}}{N_{1m}^{\frac{1-\nu}{2}}}\Tilde{\Phi}_1(0,\beta_0)<\gamma\bigg(\tfrac{1-\nu}{2},\beta_0^2\tfrac{N_{1M}}{L_{1m}}\bigg)
\end{equation}
then there exists $\tilde{\alpha_0}=\tilde{\alpha_0}(\beta_0)$ such that $\epsilon_{\beta_0}(\tilde{\alpha_0})=1$, $\epsilon(\alpha_0)<1$ for all $\alpha_0<\tilde{\alpha_0}$.
\end{lemma}
\begin{proof} We have
$$\epsilon_{\beta_0}(0)=\dfrac{2L_{1M}^{\frac{5-\nu}{2}}L_{1m}^{\nu-2}\Tilde{\Phi}_1(0,\beta_0)}{N_{1m}^{\frac{1-\nu}{2}}\gamma\bigg(\tfrac{1-\nu}{2},\beta_0^2\tfrac{N_{1M}}{L_{1m}}\bigg)}$$
and
$$\lim\limits_{\alpha_0\to \beta_0^{-}}\epsilon_{\beta_0}(\alpha_0)= \lim\limits_{\alpha_0\to \beta_0^{-}} \dfrac{2L_{1M}^{\frac{5-\nu}{2}}L_{1m}^{\nu-2}\Tilde{\Phi}_1(\alpha_0,\beta_0)}{N_{1m}^{\frac{1-\nu}{2}}\bigg[\gamma\bigg(\tfrac{1-\nu}{2},\beta_0^2\tfrac{N_{1M}}{L_{1m}}\bigg)-\gamma\bigg(\tfrac{1-\nu}{2},\alpha_0^2\tfrac{N_{1M}}{L_{1m}}\bigg)\bigg]}=+\infty $$ 
Then, if $\epsilon_{\beta_0}(0)<1$, this is \eqref{des} holds, we have there exists $\tilde{\alpha_0}<\beta_0$ such that $\epsilon_{\beta_0}(\tilde{\alpha_0})=1$ and $\epsilon_{\beta_0}(\alpha_0)<1$ for all $0<\alpha_0<\tilde{\alpha_0}$

\end{proof}
\begin{lemma} There exists $\hat{\beta_0}$ such that $\eqref{des}$ holds for $0<\beta_0<\hat{\beta_0}$
\end{lemma}
\begin{proof}
We have $$\Tilde{\Phi}_1(0,\beta_0)=\dfrac{1}{L_{1m}^2}\bigg(\bigg(\tilde{N_1}+\dfrac{N_{1M}\tilde{L_1}}{L_{1m}}\bigg)\dfrac{\beta_0^{3-\nu}}{3-\nu}+\tilde{L_1}\dfrac{\beta_0^{1-\nu}}{1-\nu}\bigg)$$ is an increasing function such that
$\Tilde{\Phi}_1(0,0)=0$, $\Tilde{\Phi}_1(0,+\infty)=+\infty$ and
 $\frac{d\Tilde{\Phi}_1}{d\beta_0}(0,0)=0.$\\
 For other hand $\gamma(\beta_0):=\gamma\bigg(\tfrac{1-\nu}{2},\beta_0^2\tfrac{N_{1M}}{L_{1m}}\bigg)$ is an increasing function in $\beta_0$ such that $\gamma(0)=0$, $\lim\limits_{\beta_0\to +\infty}\gamma(\beta_0)=\Gamma(\tfrac{1-\nu}{2})$ and $\frac{d\gamma}{d\beta_0}(0)=+\infty.$ Then there exists $\hat{\beta_0}$ such that $\eqref{des}$ holds for $0<\beta_0<\hat{\beta}_0$.

\end{proof}

\begin{lemma}
If \begin{equation}\label{hipbeta}
    \tilde{\beta_0}<\hat{\beta_0}
\end{equation} there exists 
$\tilde{\alpha_0}=\tilde{\alpha_0}(\beta_0)<\beta_0$, for $\tilde{\beta_0}<\beta_0<\hat{\beta_0}$ such that $\epsilon(\alpha_{0},\beta_0)<1$ and $\sigma(\beta_0)<1$ for all  $0<\alpha_0<\tilde{\alpha_0}(\beta_0)$.
\end{lemma}

Next, we assume $\eqref{hipbeta}$ and consider the set $\Pi=\big\lbrace{(\alpha_0,\beta_0)/\alpha_0<\tilde{\alpha_0}(\beta_0),\tilde{\beta_0}<\beta_0<\hat{\beta_0}\big\rbrace}$. Let's to solve the system of equations given by $\eqref{eq29*},\eqref{eq29}$ for $(\alpha_0,\beta_0)\in \Pi$ and $u_1=u_1(\alpha_0,\beta_0)$, $u_2=u_2(\beta_0)$ are the solutions given in Theorems 3.5 and 3.6.

Taking into account $\eqref{eq29*}$ we obtain $\eqref{eq29}$ becomes
\begin{equation}\label{eqalpha}
   \dfrac{2l_b\gamma_b(\alpha_0)^{\nu+1}}{\theta_b-\theta_m}+\dfrac{u_c }{\Phi_2[\beta_0,+\infty,u_2]}=\dfrac{2l_m\gamma_m(\beta_0)^{\nu+1}}{\theta_b-\theta_m}
\end{equation} then we can write
\begin{equation}\label{eqalpha}
\alpha_0^{*}=\alpha_0^{*}(\beta_0)=\bigg(\frac{\theta_b-\theta_m}{2l_b\gamma_b}\bigg[\dfrac{2l_m\gamma_m(\beta_0)^{\nu+1}}{\theta_b-\theta_m}-\dfrac{u_c }{\Phi_2[\beta_0,+\infty,u_2]}\bigg]\bigg)^{\frac{1}{{\nu+1}}}
\end{equation}
If we replace in \eqref{eq29*} we obtain that $\beta_0$ must satisfies 
\begin{equation}\label{eq29prima}
    \dfrac{1}{\Phi_1[\alpha_0^{*}(\beta_0),\beta_0, u_1]}-\dfrac{2l_b\gamma_b(\alpha^{*}_0(\beta_0))^{\nu+1}}{\theta_b-\theta_m}=0
\end{equation}

Let functions
\begin{equation}
J_1(\beta_0):=\frac{1}{i_2(\beta_0)}-\bigg(B^{*}\beta_0^{\nu+1}-\frac{u_c}{M(\beta_0)}\bigg)^{\tfrac{1}{\nu+1}}, \quad
J_2(\beta_0):=\frac{1}{i_1(\beta_0)}-B^{*}\beta_0^{\nu+1}.
\end{equation} where $B^{*}=\frac{2l_m\gamma_m}{\theta_b-\theta_m}$ and 
$$M(\beta_0):=\dfrac{1}{L_{2M}}\exp\bigg(\tfrac{2\beta_0^{\xi}N_{2M}}{L_{2m}\xi}\bigg)\bigg(\tfrac{L_{2m}} {2N_{2M}}\bigg)^\frac{\delta-3\mu+5}{2+\delta-\mu}\xi^{\tfrac{1-\delta-\nu}{\xi}} \bigg[\Gamma\bigg(\tfrac{3-\nu-\mu}{\xi}\bigg)-\gamma\bigg(\tfrac{3-\nu-\mu}{\xi},\tfrac{2N_{2M}\beta_0^{\xi}}{L_{2m}\xi}\bigg)\bigg]$$
$$i_1(\beta_0):=\dfrac{1}{2L_{1M}}\exp\bigg(\tfrac{N_{1M}}{L_{1m}}\alpha_0^{*2}\bigg)\sqrt{\dfrac{N_{1M}^{\nu-1}}{L_{1m}^{\nu-1}}}\bigg[\gamma\bigg(\dfrac{1-\nu}{2},\beta_0^2\dfrac{N_{1M}}{L_{1m}}\bigg)-\gamma\bigg(\dfrac{1-\nu}{2},  \frac{N_{1M}}{L_{1m}}\alpha_0^{*2}\bigg)\bigg],$$

$$i_2(\beta_0):=\frac{1}{2L_{1m}}\exp\bigg(\frac{N_{1m}}{L_M}\alpha_0^{*2}\bigg)\sqrt{\frac{N_{1m}^{\nu-1}}{L_{1M}^{\nu-1}}}\bigg[\gamma\bigg(\frac{1-\nu}{2},\beta_0^2\frac{N_{1m}}{L_{1M}}\bigg)-\gamma\bigg(\frac{1-\nu}{2},\frac{N_{1m}}{L_{1M}}\alpha_0^{*2}\bigg)\bigg].$$
\begin{lemma}
If 
\begin{equation}\label{jota}
sign(J_i(\hat{\beta}_0))\neq sign(J_i(\Tilde{\beta}_0))\end{equation} 
for $i=1,2$ we obtain there exists $\beta^{*}_0\in (\Tilde{\beta}_0,\hat{\beta}_0)$ such that \eqref{eq29prima} holds.
\end{lemma}
\begin{proof}

 On the one hand, by Lemma 3.2a) we have to
$$i_1(\beta_0)\leq\Phi_1[\alpha_0^{*},\beta_0, u_1]\leq i_2(\beta_0)$$ and the order hand, taking into account \eqref{eqalpha} and  Lemma 3.2 b) we have to
$$\bigg(\frac{B^{*}}{A^{*}}\bigg)^{\tfrac{1}{\nu+1}}\leq\alpha_0^{*}(\beta_0)\leq \bigg(\dfrac{1}{A^{*}}\bigg)^{\tfrac{1}{\nu+1}}\bigg(B^{*}\beta_0^{\nu+1}-\frac{u_c}{M(\beta_0)}\bigg)^{\tfrac{1}{\nu+1}}$$ where 
$A^{*}=\frac{2l_b\gamma_b}{\theta_b-\theta_m}$, $B^{*}=\frac{2l_m\gamma_m}{\theta_b-\theta_m}$.

Then $$J_1(\beta_0)\leq \dfrac{1}{\Phi_1[\alpha_0^{*}(\beta_0),\beta_0, u_1]}-\dfrac{2l_b\gamma_b(\alpha^{*}_0(\beta_0))^{\nu+1}}{\theta_b-\theta_m}\leq J_2(\beta_0).$$ 
If $sign(J_i(\hat{\beta}_0))\neq sign(J_i(\Tilde{\beta}_0))$ for $i=1,2$ we obtain there exists $\beta^{*}_0\in (\Tilde{\beta}_0,\hat{\beta}_0)$ such that \eqref{eq29prima} holds.

\end{proof}
From Lemmas 3.9 and 3.10 we have there exists $\tilde{\alpha_0}(\beta^{*}_0)<\beta^{*}_0$, such that $\epsilon(\alpha_{0},\beta^{*}_0)<1$ and $\sigma(\beta^{*}
_0)<1$ for all  $0<\alpha_0<\tilde{\alpha_0}(\beta^{*}_0)$. Assuming hypothesis Lemma 3.11 and Lemma 3.12 it remains to prove that $\alpha_0^{*}(\beta^{*}_0)$ given by \eqref{eqalpha} satisfies $\alpha_0^{*}(\beta^{*}_0)<\tilde{\alpha_0}(\beta^{*}_0)$.
\begin{lemma}
If 
\begin{equation}\label{desi}
    \bigg(\dfrac{1}{A^{*}}\bigg)^{\tfrac{1}{\nu+1}}\bigg(B^{*}(\beta^{*}_0)^{\nu+1}-\frac{u_c}{M(\beta^{*}_0)}\bigg)^{\tfrac{1}{\nu+1}}<\tilde{\alpha_0}(\beta^{*}_0)
\end{equation} we have $\alpha_0^{*}(\beta^{*}_0)<\tilde{\alpha_0}(\beta^{*}_0)$.
\end{lemma}
\begin{proof}
It is immediately proved by the bound given for $\alpha_0^{*}(\beta^{*}_0)$  in the previous lemma.
\end{proof}
Finally we can enunciate the following theorem
\begin{theorem}
Assuming \eqref{eq30}-\eqref{eq32}, \eqref{hipbeta}, \eqref{jota} and \eqref{desi} we obtain there exists solution to problem \eqref{eq1}-\eqref{eq1} given by $$\theta_i(r,t)=\theta_m+(\theta_b-\theta_m)u_i(\eta),\;\;i=1,2,$$ and free boundaries $$\alpha(t)=2\alpha_0^{*}(\beta^{*}_0)\sqrt{t}, \quad \beta(t)=2\beta_0^{*}\sqrt{t},\quad t>0$$
where $\eta=r/(2\sqrt{t})$ and  functions $$u_1(\eta)=u_1(\alpha_0^{*}(\beta^{*}_0),(\beta^{*}_0))(\eta),\quad \eta\in[\alpha_0^{*}(\beta^{*}_0),\beta^{*}_0],\quad u_2(\eta)=u_2(\beta^{*}_0)(\eta),\quad \eta\in[\beta^{*}_0,+\infty)$$ satisfy the nonlinear integral equations \eqref{eq27} and \eqref{eq28}, and $(\alpha_0^{*}(\beta^{*}_0),(\beta^{*}_0))$ is solution to  \eqref{eq29*},\eqref{eq29}.
\end{theorem}
\begin{proof}
From above lemmas we obtain there exists a pair $(\alpha_0^{*}(\beta^{*}_0),\beta^{*}_0)\in\Pi$ that is solution of $\eqref{eq29*},\eqref{eq29}$ and $\epsilon(\alpha_0^{*}(\beta^{*}_0),\beta^{*}_0)<1$, $\sigma(\beta^{*}_0)<1$.
Then taking into account Th 3.5 and Th 3.6 there exists $u_1(\alpha_0^{*}(\beta^{*}_0),\beta^{*}_0)$, $u_2(\beta^{*}_0)$ solutions of $\eqref{eq27},\eqref{eq28}$.
By using transformations \eqref{eq9} and \eqref{eq17} we obtain the thesis.
\end{proof}

\section{Conclusion}
Temperature field for liquid and solid phases in infinite material with variable cross-section is modeled with generalized heat equation. The solution of the problem is obtained on the base of similarity variable \cite{15}. Moreover, we determined the temperature for two phases and free boundaries which describe the location of the boiling and melting interfaces. The existence  of the solution is provided and it is shown that obtained integral operators $V$ and $W$ are a contraction operator and solving two coupled equations for coefficients that characterize the free boundaries. This Stefan problem is very useful in electrical contact phenomena when material with variable cross-section is finite, in particular, metal bridge which connect two electrical materials and in instantaneous explosion of micro asperity the metal bridge is fully melted. In this case this problem will be useful to select which material can be better to choose for metal bridge for electrical contact processes.

\end{document}